\pdfoutput=1
\documentclass[a4paper]{article}
\usepackage[english]{babel}
\usepackage{a4wide}
\usepackage[scr=boondox,scrscaled=1.05]{mathalpha}
\usepackage{amsmath,amssymb,amsthm}
\usepackage{mathtools}
\usepackage{ifpdf}
\usepackage{cite}

\ifpdf
\usepackage[unicode,bookmarks=false,pagebackref,%
hyperindex]{hyperref}
\hypersetup{pdfstartview=FitH,
linkbordercolor=[rgb]{0.5, 1, 1},
citebordercolor=[rgb]{0.5, 1, 0.5}}
\else
\fi

\newcommand{\setA}{\mathscr{A}}
\newcommand{\setB}{\mathscr{B}}
\newcommand{\setD}{\mathscr{D}}
\newcommand{\setM}{\mathscr{M}}
\let\emptyset\varnothing

\newtheorem{theorem}{Theorem}
\newtheorem{lemma}[theorem]{Lemma}
\newtheorem*{theorexp}{Theorem A}

\newtheoremstyle{myremark}
{}
{}
{\rmfamily}
{}
{\bfseries}
{.}
{.5em}
{}

\theoremstyle{myremark}
\newtheorem{problem}[theorem]{Question}
\newtheorem{remark}[theorem]{Remark}
\newtheorem{definition}[theorem]{Definition}

\begin{document}
\date{}

\title{On convergence of infinite matrix products\\ with alternating factors from two sets of matrices}

\author{Victor Kozyakin\thanks{Kharkevich Institute for Information Transmission Problems, Russian Academy of Sciences, Bolshoj Karetny lane 19, Moscow 127051, Russia, e-mail: kozyakin@iitp.ru}}

\maketitle

\begin{abstract}
We consider the problem of convergence to zero of matrix products $A_{n}B_{n}\cdots A_{1}B_{1}$ with factors from two sets of matrices, $A_{i}\in\setA$ and $B_{i}\in\setB$, due to a suitable choice of matrices~$\{B_{i}\}$. It is assumed that for any sequence of matrices $\{A_{i}\}$ there is a sequence of matrices~$\{B_{i}\}$ such that the corresponding matrix products $A_{n}B_{n}\cdots A_{1}B_{1}$ converge to zero. We show that, in this case, the convergence of the matrix products under consideration is uniformly exponential, that is, $\|A_{n}B_{n}\cdots A_{1}B_{1}\|\le C\lambda^{n}$, where the constants $C>0$ and $\lambda\in(0,1)$ do not depend on the sequence $\{A_{i}\}$ and the corresponding sequence $\{B_{i}\}$. Other problems of this kind are discussed and open questions are formulated.
\medskip

\medskip

\end{abstract}

\section{Introduction}
Denote by $\setM(p,q)$ the space of matrices of dimension $p\times q$ with real elements and the topology of elementwise convergence. Let $\setA\subset\setM(N,M)$ and $\setB\subset\setM(M,N)$ be finite sets of matrices. We will be interested in the question of whether it is possible to ensure the convergence to zero of matrix products,
\begin{equation}\label{E:matprod}
A_{n}B_{n}\cdots A_{1}B_{1},\qquad A_{i}\in\setA,~B_{i}\in\setB,
\end{equation}
for all possible sequences of matrices $\{A_{i}\}$, due to a suitable choice of sequences of matrices $\{B_{i}\}$.

As an example of a problem in which such a question arises, let us consider one of the varieties of the stabilizability problem for discrete-time switching linear systems~\cite{Stanford:SIAMJCO79,Koz:AiT90:10:e,SunGe05,JunMas:SIAMJCO17,LinAnt:IEEETAC09}. Consider a system whose dynamics is described by the equations
\begin{equation}\label{E:switchsys}
\begin{aligned}
x(n)&=A_{n}u(n),\qquad &A_{n}\in\setA,\\
u(n)&=B_{n}x(n-1),\qquad &B_{n}\in\setB,
\end{aligned}
\end{equation}
where the first of them describes the functioning of a plant, whose properties uncontrollably affected by perturbations from the class $\setA$, while the second equation describes the behavior of a controller. Then, by choosing a suitable sequence of controls $\{B_{n}\in\setB\}$, one can try to achieve the desired behavior of system~\eqref{E:switchsys}, for example, the convergence to zero of its solutions:
\[
x(n)=A_{n}B_{n}\cdots A_{1}B_{1}x(0).
\]

As was noted, for example, in~\cite{ACDDHK:STACS15,BMRLL:AI16}, the question of the stabilizability of matrix products with alternating factors from two sets, due to a special choice of factors from one of these sets, can also be treated in the game-theoretic sense.

If, in considering the switching system, it is assumed that there are actually no control actions, that is, $B_{n}\equiv I$, then~\eqref{E:switchsys} take the form
\begin{equation}\label{E:stabsys}
x(n)=A_{n}x(n-1),\qquad A_{n}\in\setA.
\end{equation}
In this case, the problem of the stabilizability of the corresponding switching system  turns into the problem of its stability for all possible perturbations of the plant in class $\setA$, that is, into the problem of convergence to zero of the solutions
\[
x(n)=A_{n}\cdots A_{1}x(0)
\]
of~\eqref{E:stabsys} for all possible sequences of matrices $\{A_{i}\in\setA\}$. Convergence to zero of the matrix products $A_{n}\cdots A_{1}$, arising in this case, has been investigated by many authors (see, e.g.,~\cite{Koz:AiT90:10:e,DaubLag:LAA92,BerWang:LAA92,Gurv:LAA95,Hartfiel:02} as well as the bibliography in~\cite{Koz:IITP13}).

The presence of alternating factors in the products of matrices~\eqref{E:matprod} substantially complicates the problem of convergence of the corresponding matrix products for all possible sequences of matrices $\{A_{i}\in\setA\}$ due to a suitable choice of sequences of matrices $\{B_{i}\in\setB\}$ in comparison with the problem of convergence of matrix products $A_{n}\cdots A_{1}$ for all possible sequences of matrices $\{A_{i}\in\setA\}$. A discussion of the arising difficulties can be found, for example, in~\cite{Koz:DCDSB19}. One of the applications of the results obtained in this paper for analyzing the new concept of the so-called minimax joint spectral radius is also described there.

\section{Path-Dependent Stabilizability}
Every product~\eqref{E:matprod} is a matrix of dimension $N\times N$; that is, it is an element of the space $\setM(N,N)$. As is known, the space $\setM(N,N)$ with the topology of elementwise convergence is normable; therefore we assume that $\|\cdot\|$ is some norm in it. We note here that since all norms in the space $\setM(N,N)$ are equivalent, the choice of a particular norm when considering the convergence of products~\eqref{E:matprod} is inessential. Nevertheless, in what follows, it will be convenient for us to assume that the norm $\|\cdot\|$ in $\setM(N,N)$ is submultiplicative; that is, for any two matrices $X,Y$, the inequality $\|XY\|\le\|X\|\cdot\|Y\|$ holds. In particular, a norm on $\setM(N,N)$ is submultiplicative if it is generated by some vector norm on $\mathbb{R}^{N}$; that is, its value on matrix $A$ is defined by the equality $\|A\|=\sup_{x\neq0}\frac{\|Ax\|}{\|x\|}$, where $\|x\|$ and $\|Ax\|$ are the norms of the corresponding vectors in $\mathbb{R}^{N}$.

\begin{definition}\label{D:PDS}
The matrix products~\eqref{E:matprod} are said to be \emph{path-dependent stabilizable} by choosing the factors $\{B_{n}\}$ if for any sequence of matrices $\{A_{n}\in\setA\}$ there exists a sequence of matrices $\{B_{n}\in\setB\}$ for which
\begin{equation}\label{E:converge}
\|A_{n}B_{n}\cdots A_{1}B_{1}\|\to0\quad\text{as}\quad n\to\infty.
\end{equation}
\end{definition}

As an example, consider the case where sets $\setA$ and $\setB$ consist of square matrices of dimension $N\times N$, and $\setB=\{I\}$, where $I$ is the identical matrix. In this case, Definition~\ref{D:PDS} of the path-dependent stabilizability of the matrix products~\eqref{E:matprod} reduces to the following condition:
\begin{equation}\label{E:convergeA}
\|A_{n}\cdots A_{1}\|\to0\quad\text{as}\quad n\to\infty,
\end{equation}
for each sequence $\{A_{n}\in\setA\}$. As is known, in this case, convergence~\eqref{E:convergeA} is uniformly exponential. Namely, the following statement, which was repeatedly `discovered' by many authors, is true (see, e.g.,~\cite{Koz:AiT90:10:e,DaubLag:LAA92,BerWang:LAA92,Gurv:LAA95,Hartfiel:02}).

\begin{theorexp}[on exponential convergence]\label{T:expconv}
Let the set of matrices $\setA$ be such that for each sequence $\{A_{n}\in\setA\}$ the convergence~\eqref{E:convergeA} holds. Then there exist constants $C>0$ and $\lambda\in(0,1)$ such that
\[
\|A_{n}\cdots A_{1}\|\le C\lambda^{n},\qquad n=1,2,\ldots\,,
\]
for each sequence $\{A_{n}\in\setA\}$.
\end{theorexp}

Our goal is to prove that an analogue of Theorem A (on exponential convergence) is valid for the path-dependent stabilizable matrix products~\eqref{E:matprod}.

\begin{theorem}\label{T:main}
Let $\setA$ and $\setB$ be the sets of matrices for which the matrix products~\eqref{E:matprod} are path-dependent stabilizable. Then there exist constants $C>0$ and $\lambda\in(0,1)$ such that for any sequence of matrices $\{A_{n}\in\setA\}$ there is a sequence of matrices $\{B_{n}\in\setB\}$ for which
\begin{equation}\label{E:convergeE}
\|A_{n}B_{n}\cdots A_{1}B_{1}\|\le C\lambda^{n},\qquad n=1,2,\ldots\,.
\end{equation}
\end{theorem}

To prove the theorem, we need the following auxiliary assertion.

\begin{lemma}\label{L:uniconv}
Let the conditions of Theorem~\ref{T:main} be satisfied. Then there exist constants $k_{*}>0$ and $\mu\in(0,1)$ such that for any sequence of matrices $\{A_{n}\in\setA\}$ there is a positive integer $k\le k_{*}$ and a set of matrices $B_{1},\ldots,B_{k}\in\setB$ for which $\|A_{k}B_{k}\cdots A_{1}B_{1}\|\le\mu<1$.
\end{lemma}

\begin{proof}
By Definition~\ref{D:PDS} of the path-dependent stabilizability of the matrix products~\eqref{E:matprod} for each matrix sequence $\{A_{n}\in\setA\}$ there exists a natural $k$ such that
\begin{equation}\label{E:p1}
\|A_{k}B_{k}\cdots A_{1}B_{1}\|<1,
\end{equation}
for some sequence of matrices $\{B_{n}\in\setB\}$.

Given a sequence $\{A_{n}\}$, let us denote by $k(\{A_{n}\})$ the smallest $k$ under which inequality~\eqref{E:p1} holds. To prove the lemma, it suffices to show that the quantities $k(\{A_{n}\})$ are uniformly bounded, that is, there is a $k_{*}$ such that
\begin{equation}\label{E:p2}
k(\{A_{n}\})\le k_{*},\qquad \forall~\{A_{n}\in\setA\}.
\end{equation}

Assuming that inequality~\eqref{E:p2} is not true, for each positive integer $k$, we can find a sequence $\{A_{n}^{(k)}\in\setA\}$ such that
$k(\{A_{n}^{(k)}\})\ge k$. In this case, by the definition of the number $k(\{A_{n}\})$,
\begin{equation}\label{E:p4}
\|A_{m}^{(k)}B_{m}\cdots A_{1}^{(k)}B_{1}\|\ge 1,\qquad\forall~B_{1},\ldots,B_{m}\in\setB.
\end{equation}
for each positive integer $m\le k-1\le k(\{A_{n}^{(k)}\})-1$.

Let us denote by $\boldsymbol{A}_{k}$ the set of all sequences $\{A_{n}\in\setA\}$, for each of which inequalities~\eqref{E:p4} hold. Then $\{A_{n}^{(k)}\}\in\boldsymbol{A}_{k}$ and, therefore, $\boldsymbol{A}_{k}\neq\emptyset$. Moreover,
\begin{equation}\label{E:p5}
\boldsymbol{A}_{1}\supseteq \boldsymbol{A}_{2}\supseteq\cdots\,,
\end{equation}
and each set $\boldsymbol{A}_{k}$ is closed since inequalities~\eqref{E:p4} hold for all its elements, sequences $\{A_{n}\}\in\boldsymbol{A}_{k}$, for each positive integer $m\le k-1$.

We now note that each of the sets $\boldsymbol{A}_{k}$ is a subset of the topological space $\setA^{\infty}$ of all sequences $\{A_{n}\in\setA\}$ with the topology of infinite direct product of the finite set of matrices $\setA$. By the Tikhonov theorem in this case $\setA^{\infty}$ is a compact. Then, each of the sets $\boldsymbol{A}_{k}$ is also a compact. In this case it follows from~\eqref{E:p5} that $\bigcap_{k=1}^{\infty}\boldsymbol{A}_{k}\neq\emptyset$ and, therefore, there is a sequence $\{\Bar{A}_{n}\in\setA\}$ such that
\[
\{\Bar{A}_{n}\}\in\bigcap_{k=1}^{\infty}\boldsymbol{A}_{k}.
\]
By the definition of the sets $\boldsymbol{A}_{k}$, for the sequence $\{\Bar{A}_{n}\in\setA\}$ the inequalities
\[
\|\Bar{A}_{m}B_{m}\cdots \Bar{A}_{1}B_{1}\|\ge 1
\]
hold for each $m\ge 1$ and any $B_{1},\ldots,B_{m}\in\setB$ which contradicts the assumption of the path-dependent stabilizability of the matrix products~\eqref{E:matprod}. This contradiction completes the proof of the existence of a number $k_{*}$ for which inequalities~\eqref{E:p2} are valid.

Thus, we have proven the existence of a number $k_{*}$ such that, for each sequence$\{A_{n}\in\setA\}$ and some corresponding sequence $\{B_{n}\in\setB\}$, strict inequalities~\eqref{E:p1} are satisfied with $k=k(\{A_{n}\})\le k_{*}$. Moreover, since the number of all products $A_{k}B_{k}\cdots A_{1}B_{1}$ participating in inequalities~\eqref{E:p1} is finite, then the corresponding inequalities~\eqref{E:p1} can be strengthened: there is a $\mu\in(0,1)$ such that for any sequence of matrices $\{A_{n}\in\setA\}$ there exist a natural $k\le k_{*}$ and a set of matrices $B_{1},\ldots,B_{k}\in\setB$ for which $\|A_{k}B_{k}\cdots A_{1}B_{1}\|\le\mu<1$.
\end{proof}

We now proceed directly to the proof of Theorem~\ref{T:main}.

\begin{proof}[Proof of Theorem~\ref{T:main}]
Given an arbitrary sequence $\{A_{n}\in\setA\}$, by Lemma~\ref{L:uniconv}, there exist a number $k_{1}\le k_{*}$ and a set of matrices $B_{1},\ldots,B_{k_{1}}$ such that
\[
\|A_{k_{1}}B_{k_{1}}\cdots A_{1}B_{1}\|\le\mu<1.
\]

Next, consider the sequence of matrices $\{A_{n}\in\setA,~n\ge k_{1}+1\}$ (the ``tail'' of the sequence $\{A_{n}\in\setA\}$ starting with the index $k_{1}+1$). Again, by virtue of Lemma~\ref{L:uniconv}, there exist a $k_{2}\le k_{1}+k_{*}$ and a set of matrices $B_{k_{1}+1},\ldots,B_{k_{2}}$ such that
\[
\|A_{k_{2}}B_{k_{2}}\cdots A_{k_{1}+1}B_{k_{1}+1}\|\le\mu<1.
\]

We continue in the same way constructing for each $m=3,4,\ldots$ numbers
\begin{equation}\label{E:p71}
k_{m}\le k_{m-1}+k_{*}
\end{equation}
and sets of matrices $B_{k_{m-1}+1},\ldots,B_{k_{m}}$ for which
\begin{equation}\label{E:p8}
\|A_{k_{m}}B_{k_{m}}\cdots A_{k_{m-1}+1}B_{k_{m-1}+1}\|\le\mu<1.
\end{equation}

Let us show that, for the obtained sequence of matrices $\{B_{n}\}$ for some $C>0$ and  $\lambda\in(0,1)$, which do not depend on the sequences $\{A_{n}\}$ and $\{B_{n}\}$, inequalities~\eqref{E:convergeE} are valid. Fix a positive integer $n$ and specify for it a number $p=p(n)$ such that
\begin{equation}\label{E:p9}
n-k_{*}<k_{p}\le n.
\end{equation}
Such $p$ exists, since the sequence $\{k_{m}\}$ strictly increases by construction. We now represent the product $A_{n}B_{n}\cdots A_{1}B_{1}$ in the form
\[
A_{n}B_{n}\cdots A_{1}B_{1}=D_{*}D_{p}\cdots D_{1},
\]
where
\begin{align*}
D_{*}&=A_{n}B_{n}\cdots A_{k_{p}+1}B_{k_{p}+1},\\
D_{i}&=A_{k_{i}}B_{k_{i}}\cdots A_{k_{i-1}+1}B_{k_{i-1}+1},\qquad i=1,2,\ldots,p.
\end{align*}

Then
\[
\|D_{*}\|\le \varkappa^{n-k_{p}}\le \varkappa^{k_{*}},\quad\text{where}\quad \varkappa=\max_{A\in\setA, B\in\setB}\{1,\|AB\|\}
\]
(since the sets $\setA$ and $\setB$ are finite, $\varkappa<\infty$). Further, by the definition of the matrices $D_{i}$ and inequalities~\eqref{E:p8},
\[
\|D_{i}\|\le \mu<1\quad\text{for}\quad i=1,2,\ldots,p.
\]
Taking into account the fact that by virtue of~\eqref{E:p71}, for each $m$, the estimate $k_{m}\le k_{*}m$ is fulfilled, from here and from~\eqref{E:p9} we obtain for the number $p$ a lower estimate: $p\ge\frac{n}{k_{*}}-1$. And then from the estimates established earlier for $\|D_{*}\|$, $\|D_{1}\|,\ldots,\|D_{m}\|$ we deduce that
\[
\|A_{n}B_{n}\cdots A_{1}B_{1}\|\le \|D_{*}\|\cdot\|D_{p}\|\cdots\|D_{1}\| \le \varkappa^{k_{*}} \mu^{p}\le \varkappa^{k_{*}} \mu^{\frac{n}{k_{*}}-1}\le\frac{\varkappa^{k_{*}}}{\mu} \left(\mu^{\frac{1}{k_{*}}}\right)^{n}.
\]
Hence, putting $C=\frac{\varkappa^{k_{*}}}{\mu}$, $\lambda= \mu^{\frac{1}{k_{*}}}$, we obtain inequalities~\eqref{E:convergeE}.
\end{proof}

\section{Path-Independent Stabilizability}
Let us now consider another variant of the stabilizability of matrix products~\eqref{E:matprod} due to a suitable choice of matrices~$\{B_{i}\}$.

\begin{definition}\label{D:PIPS}
The matrix products~\eqref{E:matprod} are said to be \emph{path-independent periodically stabilizable} by choosing the factors $\{B_{n}\}$ if there exists a periodic sequence of matrices $\{\Bar{B}_{n}\in\setB\}$ such that
\begin{equation}\label{E:convergeU}
\|A_{n}\Bar{B}_{n}\cdots A_{1}\Bar{B}_{1}\|\to0\quad\text{as}\quad n\to\infty,
\end{equation}
for any sequence of matrices $\{A_{n}\in\setA\}$.
\end{definition}

It is clear that path-independent periodically stabilized products~\eqref{E:matprod} are path-dependent stabilized.

\begin{theorem}\label{T:main2}
Let $\setA$ and $\setB$ be the sets of matrices for which the matrix products~\eqref{E:matprod} are path-independent periodically stabilizable by a sequence of matrices $\{\Bar{B}_{n}\in\setB\}$. Then there exist constants $C>0$ and $\lambda\in(0,1)$ such that
\begin{equation}\label{E:convergeEU}
\|A_{n}\Bar{B}_{n}\cdots A_{1}\Bar{B}_{1}\|\le C\lambda^{n},\qquad n=1,2,\ldots\,,
\end{equation}
for any sequence of matrices $\{A_{n}\in\setA\}$.
\end{theorem}

\begin{proof}
Denote by $p$ the period of the sequence $\{\Bar{B}_{n}\}$. Consider the set of $(N\times N)$-matrices:
\[
\setD=\{D=A_{p}\Bar{B}_{p}\cdots A_{1}\Bar{B}_{1}:~A_{1}\ldots A_{p}\in\setA\}.
\]

Since the set of matrices $\setA$ is finite, set $\setD$ is also finite. Moreover, by Definition~\ref{D:PIPS} of path-independent periodic stabilization,
\[
\|A_{np}\Bar{B}_{np}\cdots A_{1}\Bar{B}_{1}\|\to0\quad\text{as}\quad n\to\infty,
\]
for each sequence $\{A_{n}\in\setA\} $. Hence, for each sequence
$\{D_{n}\in\setD\}$, there is also
\[
\|D_{n}\cdots D_{1}\|\to0\quad\text{as}\quad n\to\infty.
\]
In this case, by Theorem A (on exponential convergence), there are $k_{*}>0$ and $\mu\in(0,1)$ such that
\[
\|D_{k_{*}}\cdots D_{1}\|\le\mu<1,\quad \forall~ D_{1}\ldots D_{k_{*}}\in\setD,
\]
or, equivalently,
\begin{equation}\label{E:last}
\|A_{k_{*}p}\Bar{B}_{k_{*}p}\cdots A_{1}\Bar{B}_{1}\|\le\mu<1,\quad \forall~A_{1}\ldots A_{k_{*}p}\in\setA.
\end{equation}

Further, repeating the proof of the corresponding part of Theorem~\ref{T:main} word for word, we derive from inequalities~\eqref{E:last} the existence of constants $C>0$ and $\lambda\in(0,1)$ such that for any sequence of matrices $\{A_{n}\in\setA\}$ inequalities~\eqref{E:convergeEU} hold.
\end{proof}

\section{Remarks and Open Questions}

First of all, we would like to make the following remarks.

\begin{remark}\label{R:0}
In the proof of Lemma~\ref{L:uniconv}, in fact, we used not the condition of path-dependent stabilizability of the matrix products~\eqref{E:matprod} but the weaker condition that for each matrix sequence $\{A_{n}\in\setA\}$ there exist a natural $k=k(\{A_{n}\})$ and a collection of matrices $B_{1},\ldots,B_{k}\in\setB$ for which equality~\eqref{E:p1} holds. Correspondingly, the statement of Theorem~\ref{T:main} is valid under weaker assumptions.

\begin{theorem}\label{T:mainX}
Let the sets of matrices $\setA$ and $\setB$ be such that for each matrix sequence $\{A_{n}\in\setA\}$ there are a natural $k$ and a collection of matrices $B_{1},\ldots,B_{k}\in\setB$ for which
\[
\|A_{k}B_{k}\cdots A_{1}B_{1}\|<1.
\]
Then there exist constants $C>0$ and $\lambda\in(0,1)$ such that for any sequence of matrices $\{A_{n}\in\setA\}$ there is a sequence of matrices $\{B_{n}\in\setB\}$ for which
\[
\|A_{n}B_{n}\cdots A_{1}B_{1}\|\le C\lambda^{n},\qquad n=1,2,\ldots\,.
\]
\end{theorem}
\end{remark}

\begin{remark}\label{R:1}\rm
All the above statements remain valid for the sets of matrices $\setA$ and $\setB$ with complex elements.
\end{remark}

\begin{remark}\label{R:2}\rm
Throughout the paper, in order to avoid inessential technicalities in proofs, it was assumed that the sets of matrices  $\setA$ and $\setB$ are finite. In fact, all the above statements remain valid in the case when the sets of matrices $\setA$ and $\setB$ are compacts, not necessarily finite, that is, are closed and precompact.
\end{remark}

Comparing the notions of path-dependent stabilizability and path-in\-depen\-dent periodic stabilizability, one can note that in the second of them the requirement of periodicity of the sequence $\{\Bar{B}_{n}\}$ stabilizing the matrix products~\eqref{E:matprod} appeared. Therefore, the following less restrictive concept of path-independent stabilizability seems rather natural.

\begin{definition}\label{D:PIS}
The matrix products~\eqref{E:matprod} are said to be \emph{path-independent stabilizable} by choosing the factors $\{B_{n}\}$ if there is a sequence of matrices $\{\Bar{B}_{n}\in\setB\}$ such that the convergence~\eqref{E:convergeU} holds for any sequence of matrices $\{A_{n}\in\setA\}$.
\end{definition}

It is not difficult to construct an example of the sets of square matrices in which the matrix products $A_{n}\Bar{B}_{n}\cdots A_{1}\Bar{B}_{1}$ converge slowly enough, slower than any geometric progression. For this it is enough to put $\setA=\{I\}$, $\setB=\{I,\lambda I\}$, where $\lambda\in(0,1)$, and define the sequence $\{\Bar{B}\}$ so that the matrix $\lambda I$ appears in it ``fairly rare'', at positions with numbers $k^{2}$, $k=1,2,\ldots\,$.

\begin{problem}
Let the matrix products~\eqref{E:matprod} be path-independent stabilizable by choosing a certain sequence of matrices $\{\Bar{B}_{n}\in\setB\}$. Is it possible in this case to specify a sequence of matrices $\{\Tilde{B}_{n}\in\setB\}$ $($possibly different from $\{\Bar{B}_{n}\in\setB\}$$)$ and constants $C>0$ and $\lambda\in(0,1)$ such that for any sequence of matrices $\{A_{n}\in\setA\}$ for all $n=1,2,\ldots$ the inequalities $\|A_{n}\Tilde{B}_{n}\cdots A_{1}\Tilde{B}_{1}\|\le C\lambda^{n}$ will be valid?
\end{problem}

Let us consider one more issue, which is adjacent to the topic under discussion. In the theory of matrix products, the following assertion is known~\cite{Koz:AiT90:10:e,DaubLag:LAA92,BerWang:LAA92,Gurv:LAA95,Hartfiel:02}: let $\setA$ be a finite set such that for each sequence of matrices $\{A_{n}\in\setA\}$  the sequence of norms $\{\|A_{n}\cdots A_{1}\|,~n=1,2,\ldots\}$ is bounded. Then all such sequences of norms for the matrices are uniformly bounded, that is, there exists a constant $C>0$ such that
\[
\|A_{n}\cdots A_{1}\|\le C,\qquad n=1,2,\ldots\,,
\]
for each sequence of matrices $\{A_{n}\in\setA\}$.

\begin{problem}
Let finite sets of matrices $\setA$ and $\setB$ be such that for each sequence of matrices $\{A_{n}\in\setA\}$ there is a sequence of matrices $\{B_{n}\in\setB\}$ for which the sequence of norms $\{\|A_{n}B_{n}\cdots A_{1}B_{1}\|,~n=1,2,\ldots\}$ is bounded. Does there exist in this case a constant $C>0$ such that for every matrix sequence $\{A_{n}\in\setA\}$ there is a sequence of matrices $\{B_{n}\in\setB\}$, for which the sequence of norms $\{\|A_{n}B_{n}\cdots A_{1}B_{1}\|,~n=1,2,\ldots\}$ is uniformly bounded, that is, for all $n=1,2,\ldots$  the inequalities $\|A_{n}B_{n}\cdots A_{1}B_{1}\|\le C$ hold?
\end{problem}


\section*{Acknowledgments}
The work was carried out at the Kharkevich Institute for Information Transmission Problems, Russian Academy of Sciences, and was funded by the Russian Science Foundation (Project no.~14-50-00150).

\bibliographystyle{elsarticle-num-no-extra-url}
\bibliography{AltFactors}
\end{document}